\documentclass[12pt,reqno]{amsart}
\usepackage{amssymb,bm,cite,fullpage,mathabx,color,enumerate}
\usepackage{tikz,tikz-3dplot}
\usepackage[colorlinks=true,linkcolor=blue,citecolor=blue]{hyperref}

\newtheorem*{cor}{Corollary}

\theoremstyle{remark}
\newtheorem{remark}{Remark}[section]

\theoremstyle{plain}
\newtheorem{thm}{Theorem}[section]

\newcommand{\R}{{\mathbb R}}

\newcommand{\N}{{\mathbb N}}

\newcommand{\Z}{{\mathbb Z}}

\def\R{\mathbb{R}}
\def\Z{\mathbb{Z}}

\title{Discrete multilinear maximal functions and number theory}
\author[T.C. Anderson]{Theresa C. Anderson}
\address{Department of Mathematics, Carnegie Mellon University, Hammerschlag Dr., Pittsburgh, PA 15213}
\email{tanders2@andrew.cmu.edu}

\numberwithin{equation}{section}

\begin{document}

\maketitle

\begin{abstract}
    Many multilinear discrete operators are primed for pointwise decomposition; such decompositions give structural information but also an essentially optimal range of bounds.  We study the (continuous) slicing method of Jeong and Lee -- which when debuted instantly gave sharp multilinear operator bounds -- in the discrete setting.  Via several examples, number theoretic connections, pointed commentary, and a unified theory we hope that this useful technique will lead to further applications.  This work generalizes, and was inspired by, the author's work with Palsson on a special case. 
\end{abstract}

\section{Introduction}

Studying analytic operators from both a multilinear perspective and a discrete one has been an active area of research.  Typically these operators have non-trivial boundedness properties if the underlying surface of integration is curved.  A prototypical curved object is the sphere, and maximal spherical averaging operators arise naturally in many contexts.  From the multilinear view, optimal Lebesgue space bounds for (multilinear) spherical maximal functions had been pursued in many papers, such as \cite{Oberlin, BGHHO, Geba_et_al, HHYpreprint}, building upon work of Stein \cite{St76} and Bourgain \cite{B86}.  From a discrete perspective, Magyar-Stein-Wainger showed optimal bounds that were both different from the continuous ones and heavily employed number theoretic techniques \cite{MSW}.  Since their work many analogues and variations have been considered, such as \cite{Pierce, Hughes, Mirek_etal, MirekTrojan, Ionescu, ACHK} to name just a few.

Very recently, Jeong and Lee debuted a \emph{slicing} or \emph{slice and dice} technique that not only allowed one to get the sharp bounds for the bilinear spherical maximal function \cite{JL} (see also \cite{Dosidis}), but showcased a nice decomposition of the operator into a pointwise product of linear operators.  Building on this work, Anderson and Palsson employed this technique in the discrete setting to very quickly deduce optimal bounds there for multilinear versions \cite{AP2} (these are recalled in the next section).  This instantly improved their previous work \cite{AP}, which relied on number theoretic techniques, in all but two dimensions.  It remains an open question about the optimal bounds in those dimensions; the difficulty of getting optimal multilinear bounds in low dimensions is indeed due to the dimensional restrictions of the linear operators, and this phenomenon does not seem to appear in the continuous setting.  However, the slicing technique provides fast and powerful results in all other dimensions, along with a nice pointwise product bound.

This technique is much more general than the multilinear spherical averaging operator.  Here we initiate the study into the generality of Jeong and Lee's slicng technique in the discrete multilinear setting.  We prove several bounds for multilinear averaging operators, connect the slicing with the underlying number theory of the arithmetic surface, and formulate a general result. 

We begin this paper with a very brief description of the slicnig technique in the discrete setting, before showing how it works via the discrete multilinear Hardy-Littlewood maximal function, the discrete multilinear $k$-spherical maximal function, and the discrete multilinear maximal function over the \emph{Waring-Goldbach} surface.  The results for the Hardy-Littlewood case are basic, but display the structure of the technique.  The spherical results are found in \cite{AP2}, so our treatment is brief.  Finally, the Waring-Goldbach operator is the most intricately connected to number theory via the well-studied \emph{Waring-Goldbach Problem}, which asks for statistics of \emph{prime} lattice points on spheres \cite{Hua_book} (more details appear in the next section).  It is here that the multilinearity and the underlying distribution of the averaging set (primes) interact the most subtly.  The Waring-Goldbach results that we obtain are indicative of this and have definite connections to how one must take care in relating the Waring-Goldbach problem in $\Z^{\ell d}$ to one in $\Z^d$.  This can be viewed in some sense as a lack of scale invariance or lack of translation-dilation invariance when working over restricted sets such as the primes.  There are likely several immediate applications to bilinear ergodic theory, specifically to multilinear operators whose linear versions have restrictions.  The slicing method provides a criterion for classifying the operators where one can easily utilize the linear ergodic theorems to piggyback to multilinear results.  While we specifically outline the Waring-Goldbach case here, without delving into these applications, it is likely that this framework is useful much more generally for other restricted sets.  Finally, based on all of these examples, we describe and discuss a general framework and result in this paper's final section.  Likely many extensions are possible; this study is only the start.

\subsection{Acknowledgements}
The author is supported by NSF DMS-1954407 and DMS-2231990.  She thanks Angel Kumchev for informative discussions about the Waring-Goldbach Problem and the anonymous referee for several suggestions that greatly improved the exposition.

\section{Introduction of technique and applications}
The basic argument of Jeong and Lee involves slicng the bilinear spherical maximal function into two linear pieces: the Hardy-Littlewood maximal function and the (linear) spherical maximal function.  From there the optimal linear bounds for these operators, properly interpolated with the symmetric slicing (i.e. switching the functions that go with each piece), and trivial estimates, allows for the optimal bounds.  

This technique is extremely natural in the discrete setting, allowing for fluid, short, proofs as exhibited in several of the examples below.  There is the additional feature, mentioned in the Introduction, that since there are sometimes dimensional restrictions on the linear operators (which arise from underlying number theoretic considerations), these amplify in the multilinear setting, missing the sharp bounds in just a few low dimensions.  However, the ease and generality of this technique, along with the wide range of sharp bounds, is appealing, and reduces difficult problems to a few exceptional cases.  Hence, when we discuss sharpness, we mean sharpness in Lebesgue space exponents; noting that the dimensional restrictions oftentimes miss the full conjectured sharp bounds.  

The three specific classes of operators that we study all display different subtleties of this method, particularly in the Waring-Goldbach case.  Bilinear cases of the first two types of operators studied are easy to depict pictorally in terms of the exponents mapped from.  For all of these cases, we get the full \emph{H\"older range}; a figure of the bilinear bounds for the operators defined in Sections 2.1 and 2.2 appear in Figure 1 at the end of Section 2.3.  A key remark is that the \emph{nesting property} (described below) of discrete spaces allows us to extend these bounds from the H\"older range to the full range of bounds.

We use the notation $|\bm{u}|^2 = |u_1|^2 + \dots |u_d|^2$ for vectors in $\Z^d$.  We will frequently and implicitly use the \emph{nesting} feature of the discrete $l^p$ norms, that is that $\|f\|_q \leq \|f\|_p$ for all $1 \leq p \leq q \leq \infty$.  we also use symmetry in our arguments to claim the overall bounds: that is, showing an $l^{p_1}(\Z^d) \times \cdots \times l^{p_\ell}(\Z^d)$ bound also shows a similar bound with any permutation of the $p_i$ indices, unless otherwise indicated.

\subsection{Discrete multilinear Hardy-Littlewood maximal function}
The $\ell$-linear discrete Hardy-Littlewood averages are defined by
\[
T_\lambda(f_1,\ldots,f_{\ell})(\bm{x}) = \left| \frac{1}{B(\lambda)}\sum_{\bm{u_1}^2+\cdots + \bm{u_\ell}^2\leq\lambda}f_1(\bm{x}-\bm{u_1})\cdots f_\ell (\bm{x}-\bm{u_\ell}) \right|
\]
where $\bm{u_i} \in \Z^d$ and 
\[
B(\lambda) = \# \{ (\bm{u_1},\dots ,\bm{u_\ell} \in \Z^{\ell d}: \sum_i|\bm{u_i}|^2\leq \lambda\}
\]
is the number of lattice points on the ball of radius $\lambda^{1/2}$ in $\R^{\ell d}$, which is asymptotic to $c_d\lambda^{\ell d/2}$ for all $d \geq 1$ (we will frequently drop dimensional constants here).
We then define the corresponding maximal operator as
\[
T^*(f_1,\ldots,f_{\ell})(\bm{x}) = \sup_{\lambda \in \N} | T_\lambda(f_1,\ldots,f_{\ell})(\bm{x}) |. 
\]
We now illustrate the \emph{slice and dice} technique to show the following (likely) well-known result.
\begin{thm}
 $T^*(f_1, \dots , f_{\ell})$ is bounded on $l^{p_1}(\Z^d)\times\ldots\times l^{p_\ell}(\Z^d) \to l^{r}(\Z^d)$, $d \geq 1$, $\sum_i\frac{1}{p_i} \geq \frac{1}{r}$, $p_i> 1$ and $r>1/\ell$.  This range is sharp.
\end{thm}
\begin{proof}
We prove this by induction.  We start with the base case $\ell = 2$.  Here we slightly rename our variables and work with
\[
T_\lambda(f,g)(\bm{x}) = \left| \frac{1}{B(\lambda)}\sum_{|\bm{u}|^2+|\bm{v}|^2\leq \lambda}f(\bm{x}-\bm{u})g(\bm{x}-\bm{v}) \right|.
\]
Thus
\[
|T_\lambda (f,g)(\bm{x})| \lesssim \frac{1}{\lambda^{\frac{d}{2}}}\sum_{|\bm{u}|^2 \leq \lambda}|f(\bm{x}-\bm{u})|\cdot \frac{1}{\lambda^{\frac{d}{2}}} \bigg|\sum_{|\bm{v}|^2 \leq \lambda - |\bm{u}|^2} g(\bm{x}-\bm{v})\bigg|.
\]
Now let $\eta = \lambda - |\bm{u}|^2$, so $\eta \leq \lambda$ and
\[
\frac{1}{\lambda^{\frac{d}{2}}} \bigg|\sum_{|\bm{v}|^2 \leq \lambda - |\bm{u}|^2} g(\bm{x}-\bm{v})\bigg| \leq \sup_{\eta \in\N}\frac{1}{\eta^{\frac{d}{2}}} \bigg|\sum_{|\bm{v}|^2 \leq \eta} g(\bm{x}-\bm{v}) \bigg| = M_{HL}(g)(\bm{x})
\]
where $M_{HL}$ is the (discrete) linear Hardy-Littlewood maximal function.   Thus we can slice our original operator and insert the above to get the pointwise product
\[
|\sup_{\lambda\in\N}|T_\lambda (f,g)(\bm{x})| \lesssim \sup_{\lambda\in\N} \frac{1}{\lambda^{\frac{d}{2}}}\sum_{|\bm{u}|^2 \leq \lambda} |f(\bm{x}-\bm{u})| \cdot M_{HL}(g)(\bm{x}) = M_{HL}(f)(\bm{x}) \cdot M_{HL}(g)(\bm{x})
\]
and since each operator in this pointwise product is bounded on $l^p(\Z^d)$ for all $p >1$, the result follows with $r > 1/2$.

Now we assume boundedness of the $(\ell-1)$-linear operator and prove bounds for the $\ell$-linear one.  We first slice
\[
T_\lambda(f_1,\ldots,f_{\ell})(\bm{x}) \lesssim \frac{1}{\lambda^{d/2}}\sum_{\bm{u_1}^2\leq\lambda}|f_1(\bm{x}-\bm{u_1})|\cdot \frac{1}{\lambda^{(\ell-1)d/2}} \left|\sum_{|\bm{u_2}|^2+ \dots |\bm{u_\ell}|^2\leq\lambda} f_2 (\bm{x}-\bm{u_2})\cdots f_\ell (\bm{x}-\bm{u_\ell}) \right|
\]
and dice
\[
\frac{1}{\lambda^{(\ell-1)d/2}}\sum_{|\bm{u_2}|^2+ \dots |\bm{u_\ell}|^2\leq\lambda} \prod_{i=2}^\ell\left|f_i (\bm{x}-\bm{u_i}) \right|
\leq M_{HL}^{\ell-1}(f_2, \dots ,f_\ell),
\]
so overall
\[
T^*(f_1, \dots , f_{\ell}) \lesssim M_{HL}(f_1)\cdot M_{HL}^{\ell-1}(f_2, \dots ,f_\ell).
\]
By the induction hypothesis, we know the $l^{p_2}(\Z^d)\times\ldots\times l^{p_\ell}(\Z^d) \to l^{r}(\Z^d)$ bounds for $M_{HL}^{\ell-1}$, therefore overall, we can conclude the claimed bounds.

We now show sharpness.  Our pointwise products obtained above suggest that we can use pointwise products of the sharpness examples for the linear version of our operator, and indeed this is the case.  We briefly outline the argument.  Let $f_i = \delta_0$ (the Dirac delta function centered at the origin).
\[
\|\sup_{\lambda\in\N}|T_\lambda (f_1, \dots , f_\ell))\|^r_{l^r(\Z^d)} =
\sum_{\bm{x} \in \Z^d}\sup_{\lambda\in\N}\bigg|\frac{1}{\lambda^{\frac{\ell d}{2}}}\sum_{|\bm{u_1}|^2+\cdots +|\bm{u_\ell}|^2 \leq  \lambda}\prod_{i=1}^\ell \delta_0(\bm{x}-\bm{u_i})\bigg|^r
\]
\[
\geq \sum_{\bm{x}\in \Z^d}\sup_{\lambda\in\N}\frac{1}{\lambda^{\ell dr/2}}\cdot \#\{\bm{u_i} : \bm{u_i} = \bm{x}, \ell|\bm{x}|^2 \leq \lambda\}.
\]
Choose $\lambda = \ell|\bm{x}|^2$ and bound from below by
\[
 \geq \sum_{\bm{x}\in\Z^d} \frac{1}{(\ell|\bm{x}|^2)^{\ell dr/2}} \geq C_{\ell, d, r}\sum_{\bm{x} \in \Z^d}\frac{1}{|\bm{x}|^{\ell dr}} 
\]
which converges if and only if $r > 1/\ell$ for all $d \geq 1$.

\end{proof}

\begin{remark}
This result can easily be upgraded to weak type restricted estimates at endpoints using known endpoint estimates for the Hardy-Littlewood maximal function.  Similar comments also apply to all results stated below, using the known endpoint estimates for the linear operators.
\end{remark}

\begin{remark}
This result also holds in the case of ``k-balls", that is, averages over surfaces $B_k := \{\sum_i |\bm{u_i}|^k \leq \lambda\}$.  The above proofs all follow through with obvious modifications, leading to the same results.
\end{remark}

\subsection{Discrete multilinear (k)-spherical maximal function}
These results are proved in \cite{AP2}.  We recall the main results, make a few comments, and refer the reader to \cite{AP2} for more details.

We begin by defining averages over ``degree $k$" spheres, $k \geq 3$, $k \in \Z$ (if $k$ is odd, we assume that $y^k = |y|^k$).
Recall that the $\ell$-linear degree $k$ discrete spherical averages are defined as:
\[
T_\lambda(f_1,\ldots,f_{\ell})(\bm{x}) = \left| \frac{1}{N(\lambda)}\sum_{\bm{u_1}^k+\cdots + \bm{u_\ell}^k=\lambda}f_1(\bm{x}-\bm{u_1})\cdots f_\ell (\bm{x}-\bm{u_\ell}) \right|
\]
where $N(\lambda) = \#\{ (\bm{u_1},\dots ,\bm{u_\ell}) \in \Z^{\ell d}: \sum_i|\bm{u_i}|^k= \lambda\}$ is asymptotic to $c_{d,k}\lambda^{\frac{\ell d}{k}-1}$ for all $d > d_0(k)/\ell$ (see \cite{ACHK2} for precise values). Define the maximal function $T^*$ as in the previous subsection.  

Also define \[r_0(d,k) = \frac{2+2\delta_0(d,k)}{(\ell - 1)(2+2\delta_0(d,k)) + (1+2\delta_0(d,k))},\] where $\delta_0(d,k)$, defined on page 2 of \cite{ACHK2}, relates to the best known bounds for the discrete linear degree $k$ operator.  Finally define \[p_0(d,k) = max \{ 1+\frac{1}{1+2\delta_0(d,k)}, \frac{d}{d-k}\}.\]  This $p_0(d,k)$ provides the best known $l^p$ bounds for the discrete linear operator; it is conjectured to be equal to $\frac{d}{d-k}$ for all $k \geq 3$.  We can now state (from \cite{AP2})
\begin{thm}
 $T^*(f_1, \dots , f_{\ell})$ is bounded on $l^{p_1}(\Z^d)\times\ldots\times l^{p_l}(\Z^d) \to l^{r}(\Z^d)$, $\frac{1}{p_1} + \ldots + \frac{1}{p_l} \geq \frac{1}{r}$, $r>\max\{ r_0(d,k), \frac{d}{\ell d-k}\}$, $p_1,\ldots,p_{\ell}> 1$ and $d > d_{0}(k)$.  Moreover, the bound $r >  \frac{d}{\ell d-k}$ is a necessary condition.
\end{thm}
The case $k=2$ is much easier to state
\begin{cor}
 $T^*(f_1, \dots , f_{\ell})$ is bounded on $l^{p_1}(\Z^d)\times\ldots\times l^{p_\ell}(\Z^d) \to l^{r}(\Z^d)$, $\frac{1}{p_1} + \ldots + \frac{1}{p_l} \geq \frac{1}{r}$, $r>\frac{d}{\ell d-2}$, $p_1,\ldots,p_{\ell}> 1$ and $d \geq 5$.
\end{cor}
Proofs of these are found in \cite{AP2}; however, we want to mention that there is an even easier argument to show the sharpness of these bounds (or necessary conditions).  We briefly sketch this argument below in the most general case
\begin{proof}
We show the necessary condition via the simpler example $f_i = \delta_0$ for all $i$.  Then 
\[
\|\sup_{\lambda\in\N}|T_\lambda (f_1, \dots , f_\ell))\|^r_{l^r(\Z^d)} = 
\sum_{\bm{x} \in \Z^d}\sup_{\lambda\in\N}\bigg|\frac{1}{\lambda^{\frac{\ell d}{k}-1}}\sum_{|\bm{u_1}|^k+\cdots +|\bm{u_\ell}|^k =  \lambda}\prod_{i=1}^\ell\delta_0(\bm{x}-\bm{u_i})\bigg|^r.
\]
As earlier choose $\lambda = \ell|\bm{x}|^k$, so that there is only one nonzero term in the inner sum for each $\bm{x}$, and we can bound
\[
 \geq \sum_{\bm{x}\in\Z^d} \frac{1}{\ell |\bm{x}|^{(\ell d -k))r}}
\]

which converges if and only if $r \geq \frac{d}{\ell d-k}$, which matches the sharp range for $k=2$.  Moreover, it works for any dimension where there is an infinite sequence of spheres such that the Hardy-Littlewood asymptotic holds (after redefining the operator to only average over the restricted range).  For example, if $k=2$, this argument works for $d=4$ whenever $\lambda$ is not divisible by 4.
\end{proof}

\begin{figure}[h!]\label{range}
\begin{tikzpicture}
\draw (0,0) rectangle (5,5);
\path[fill=blue!35] (0,0)--(0,5)--(3,5) --(5,3)--(5,0);
\draw (0,5)--(0,0)--(5,0);
\node [below left] at (0,0) {$O$};\node [above] at (0,5) {$A$};
\node [above] at (2.5,5) {$B=(\frac{d-2}{d},1)$};
\draw (3,5) circle [radius=0.04];
\node [right] at (5,3) {$B'$};
\draw (5,3) circle [radius=0.04];
\draw (5,5) circle [radius=0.04];
\node [below right] at (5,0) {$A'$};
\draw [<->] (5.8,0.7)--(5.8,0)--(6.5,0);
\node at (6.7,0) {$\frac{1}{p}$};
\node at (5.8,0.7) [right]{$\frac{1}{q}$};
\end{tikzpicture}
\caption{The range of $l^p(\Z^d)\times l^q(\Z^d)$ bounds (in terms of $\frac{1}{p}$ and $\frac{1}{q}$) for the operators in 2.1 and 2.2: the full square represents the bounds for the bilinear operator in 2.1 and the blue region for the operator in 2.2.  The notation $X'$ represents the symmetric point to $X$ about the diagonal.  These bounds are all sharp.}
\end{figure}
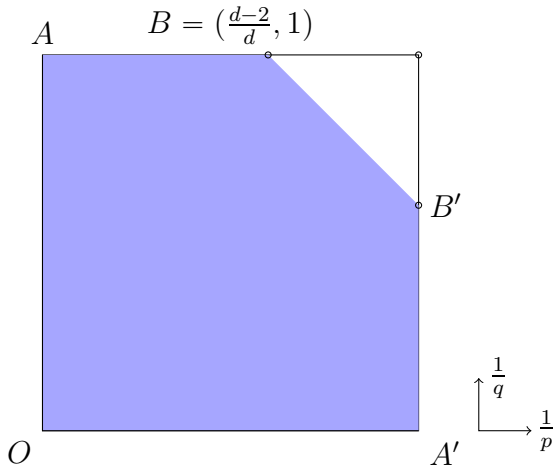

\subsection{Discrete multilinear Waring-Goldbach maximal function}
We now show an example of how this technique works over multilinear averages on ``sparser" sets via considering multilinear averages over prime points on spheres.  The delicate interplay between multilinearity, curvature, and primes demonstrates that multilinear surfaces defined over restricted subsets of the integers often require more intricate analysis.  This example also indicates the difficulty of directly relating ``restricted" solutions of Diophantine surfaces between different dimensions.

The discrete averaging operators we consider are averages over the prime vectors on the algebraic surface in $\Z^{\ell d}$
\begin{equation}
\label{sphere}
|\bm{p_1}|^k+ \dots + |\bm{p_\ell}|^k = \lambda, 
\end{equation} 
Let $P(\lambda)$ denote the number of prime solutions counted with logarithmic weights
\[ P(\lambda) = \sum_{|\bm{p_1}|^k + \dots |\bm{p_\ell}|^k = \lambda} \prod_{i=1}^\ell \log(\bm{p_i}), \] where $\log\mathbf x_i = (\log x_1) \cdots (\log x_d)$ and ${\bf p}$ is a vector with all $d$ coordinates prime.

The Waring--Goldbach problem in analytic number theory involves the study of these solutions (typically stated for $\ell = 1$).  Classic work of Hua \cite{Hua_book} gives an asymptotic (for $d$ large enough with respect to $k$ -- see \cite{Kumchev_Wooley1,Kumchev_Wooley2} for best recent results) as long as we restrict to a certain arithmetic progression $\Gamma_{d,k}$: 
\begin{equation}
\label{Hua}
P(\lambda) \sim \lambda^{{\ell d/k - 1}},
\end{equation} 
where we have ignored the constants, singular series and Gamma function factors that for our purposes can be regarded as constants.  Some examples of progressions $\Gamma_{d,k}$ are (see Chapter VIII in Hua~\cite{Hua_book} for more details): 
\begin{itemize}
\item $\Gamma_{d,k}$ is the residue class $\lambda \equiv d \pmod 2$ when $k$ is odd, $d>3k$;
\item $\Gamma_{5,2}$ is the residue class $\lambda \equiv 5 \pmod {24}$;
\item $\Gamma_{17,4}$ is the residue class $\lambda \equiv 17 \pmod {240}$.
\end{itemize} 
We'll call such progressions \emph{admissible} (a precise definition will follow below).

The authors in \cite{ACHK} prove $l^p(\Z^n)$ bounds for the linear discrete spherical maximal function along the primes.  They consider the discrete spherical averages
\begin{equation}\label{primeavg} 
A_\lambda (f)({\bf x}) := \frac{1}{P(\lambda)} \sum_{|{\bf p}|^k=\lambda} (\log{\bf p}) f({\bf x-p}). 
\end{equation}
and corresponding maximal function
\[
A^*(f)(\bm{x}) = \sup_{\lambda \in \Gamma_{d,k}}|A_\lambda (f)(\bm{x})|
\]
and show the following bounds (see \cite{ACHK} for definitions and notation -- the authors use $n$ there instead of $d$ to indicate dimension).
\begin{thm}\label{mainmaxfunction}
Let $k \geq 2$ and $d \geq \max\{ d_1(k), d_2(k) \}$. Let $\Gamma_{d,k}$ be an admissible progression.  The maximal function given by $A^*(f)$ is bounded on \(\ell^p(\Z^d)\) for all $p > p_{k,d}$. 
\end{thm}
In particular, this operator is bounded for all $p> d/(d-2)$ when $k=2$ and $d \geq 7$ (and this bound is sharp in terms of $p$).

We extend their definition to the multilinear setting, which requires a careful definition.  For ease of notation, consider the bilinear extension that mimics our earlier definitions:
\[
A_\lambda (f,g)({\bf x}) := \frac{1}{P(\lambda)} \sum_{|{\bf p}|^k+ |{\bf q}|^k=\lambda} (\log{\bf p})(\log{\bf q}) f({\bf x-p}) g({\bf x-q})
\]
and maximal function
\[
A^* (f,g) := \sup_{\lambda \in \Gamma_{2d,k}} |A_\lambda (f,g)| .
\]
It turns out that to say something meaningful about the boundedness of this operator using the slicing method, one must also impose conditions upon both $\bm{p}$ and $\bm{q}$ separately.  We first give a thorough commentary before specifying these conditions.

In the definition of $A_\lambda(f,g)$, we would like to replace $P(\lambda)$ by its asumptotic.  We have a nice asymptotic as long as $\lambda \in \Gamma_{2d,k}$ (for the $\ell$-linear version, we would need $\lambda \in \Gamma_{\ell d,k}$), and naturally this condition is translated to the supremum in the maximal function.  However, this isn't quite enough to be able to apply the slicing method.  The slicing method splits the operator into two pieces, one depending on $\bm{p}$ and one on $\bm{q}$ -- we would then like to apply the Hardy-Littlewood maximal function (along the primes) bound to the piece involving $\bm{p}$ and the linear operator $A^*(g)$ bound to the piece involving $\bm{q}$ (and also the symmetric condition switching the roles of $\bm{p}$ and $\bm{q}$).  To be able to bound $A^*(g)$ we need that $|\bm{q}|^k \in \Gamma^1_{d,k}$ for some allowable progression.  Similarly, switching the roles of the primes, we need $|\bm{p}|^k \in \Gamma^2_{d,k}$ for another allowable progression.  However, since we assume $\lambda \in \Gamma_{2d,k}$, we also need 
\[
\Gamma^1_{d,k}+\Gamma^2_{d,k} = \Gamma_{2d,k},
\]
a ``sumset" condition.  Note that the sumset condition is not enough without assuming both $|\bm{p}|^k \in \Gamma^1_{d,k}$ and $|\bm{q}|^k \in \Gamma^2_{d,k}$ (an easy example as to why is described below).  

\begin{remark}
We give an example to show that these conditions on the progression $\Gamma$ aren't always empty.  Let $k$ be odd, $\ell = 2$, $\Gamma^i_{d,k} = d \mod 2$ and $\Gamma_{\ell d,k} = 2d \mod 2$.  Now let $d$ be even and large; thus $d \equiv 2d \equiv 0 \mod 2$.  These conditions in fact imply that $|\bm{p}|^k, |\bm{q}|^k$ both are even and $\lambda$ is even.  If one of $|\bm{p}|^k, |\bm{q}|^k$ were odd, then so would $\lambda$, a contradiction.  If both of $|\bm{p}|^k, |\bm{q}|^k$ were odd, then since $d$ is even this implies that $\bm{p}$ and $\bm{q}$ have an odd number of '2' entries, which is not allowed (no repetitions are permitted).  Thus every even $\lambda$ will have solutions along the slices $|\bm{p}|^k$ and $|\bm{q}|^k$.
\end{remark}

With this is mind, we define the $\ell$-linear discrete spherical averages along the primes, hilighting the fact in the definition that prime vectors are only allowed on surfaces where the Waring-Goldbach asymptotic holds (for $d, k, \lambda$ where this is defined, as discussed above).  Firstly, let 
\[
P_\ell^{d}(\lambda) = \#\{ (\bm{p_1},\dots ,\bm{p_\ell}) \in \Z^{\ell d}: \bm{p_i}\in \Gamma_{d,k}, \sum_i|\bm{p_i}|^k= \lambda\} 
\]
\[
A_\lambda (f_1, \dots ,f_\ell)({\bf x}) := \frac{1}{\lambda^{\ell d/k -1}} \sum_{(\bm{p_1}, \dots , \bm{p_\ell}) \in P_\ell^{ d}(\lambda)} \prod_{i=1}^\ell (\log{\bf p_i}) f_i({\bf x-p_i})
\]
and corresponding maximal function (with supremum over $\lambda \in \Gamma_{\ell d,k}$.  We emphasize yet again that this maximal function includes a sum over $P_\ell^d$ which is sparser than the original Waring-Goldbach surface $P_1^{d}(\lambda) = \#\{ (\bm{p_1},\dots ,\bm{p_\ell}) \in \Z^{\ell d}: \sum_i|\bm{p_i}|^k= \lambda\}$.  For our slicing technique to work we must therefore ensure that 
\begin{equation}
\label{key_condition}
    \sum_{(\bm{p_1}, \dots , \bm{p_\ell}) \in P_\ell^{d}(\lambda)} \prod_{i=1}^\ell (\log{\bf p_i})\sim\lambda^{\ell d/k-1}, \lambda \in \Gamma = \Gamma_{d,k}^\ell = \Gamma_{d,k}^1+ \cdots \Gamma_{d,k}^\ell 
\end{equation}

Recall that we call a sequence $\Gamma_{d,k}$ admissible if $P(\lambda) \sim \lambda^{d/k-1}$.  If all $\Gamma_{d,k}^i$ are admissible, we still need \eqref{key_condition} to be satisfied.  More information on this appears shortly, but we pause to introduce the main theorem of this section.

\begin{thm}
\label{WG}
  Choose admissible $\Gamma_{d,k}^i$ and let \eqref{key_condition} be satisfied.
  Then $A^*(f_1, \dots , f_{\ell})$ is bounded on $l^{p_1}(\Z^d)\times\ldots\times l^{p_l}(\Z^d) \to l^{r}(\Z^d)$, $\frac{1}{p_1} + \ldots + \frac{1}{p_\ell} \geq \frac{1}{r}$, $p_1,\ldots,p_{\ell}> 1$ and $r>\frac{p_{k,d}}{(\ell-1)p_{k,d}+1}$.  We also have $r >\frac{d}{\ell d-k}$ is necessary.
\end{thm}

We delay the proof to comment on the number theory underlying these surfaces.

The initiation of the study of the Waring-Goldbach problem is often attributed to both Hua and Vinogradov.  Hua studied asymptotics when we restrict $\lambda$ to be in a certain arithmetic progression; the reason for this restriction is that by choosing an appropriate progression, we avoid forcing the use of small primes, which would decrease the number of points on the Waring-Goldbach surface (see \cite{JB} for more information).  While it is true that if $d$ is large enough with respect to $k$ then every $\lambda$ has a Waring-Goldbach solution, we need much more for our techniques to work, namely the same nice asymptotic in each component that Hua used and that was used to study the linear operator.  The requirement in \eqref{key_condition} is in addition necessary for the linearity to interact with the surface.  To comment further, let $\theta(k,p)\in \N$ be the largest power of $p$ dividing $k$ (set $\theta(k,p) = 0$ if there is none), and define 
\begin{equation*}
\gamma(k,p)=\begin{cases}
          \theta+2 \quad &\text{if} \, p=2, 2\divides k \\
          \theta+1 \quad &\text{otherwise}. \\
     \end{cases}
\end{equation*}
Let 
\begin{equation}
    K(k) = \prod_{(p-1)\divides k}p^{\gamma(k,p)};
\end{equation}
admissibility conditions are now determined modulo $K(k)$.  That is, for $d$ large enough with respect to $k$ \cite{Kumchev_Wooley1,Kumchev_Wooley2}, $\Gamma_{d,k}$ is admissible if it consists of residue classes $\lambda \equiv d\pmod{K(k)}$.  The cases when $d$ is smaller than this threshold are a bit more delicate, see, for example, \cite{Kum} for more information.  Since we will need to use bounds for our linear Waring-Goldbach operator, we will always assume the (more restrictive) $d \geq \max\{ d_1(k), d_2(k) \}$ from \cite{ACHK}.  It is easy to see that $K(k) = 2$ for $k$ odd.  Therefore, if $k,d,\ell$ are all odd and $\Gamma_{d,k}^i$ is the progression $1 \pmod{2}$, then \eqref{key_condition} is satisfied with $\Gamma \equiv 1 \pmod{2}$.  Similarly, if $k$ is odd, $d$ is even, and $\Gamma_{d,k}$ is the progression $0 \pmod{2}$, then \eqref{key_condition} is similarly satisfied for all $\ell$.  More generally, if
\[
S :=\{\ell d \pmod{K(k)} \cap d \pmod{K(k)}\} \neq \emptyset
\]
then \eqref{key_condition} holds.

Finally, one could study the multilinear variant of the alternative linear operator:
\[
C_\lambda (f)({\bf x}) := \frac{1}{N_{d,k,t}(\lambda)} \sum_{|{\bf p}|^k=\lambda} (\log{\bf p}) f({\bf x-p}). 
\]
where $N_{d,k,t}(\lambda) = \lambda^{d-1-t}(\log \lambda)^{-d}$ is a varying asymptotic with $t$, whose value changes depending on the residue class $d \pmod{K(k)}$.  For instance, if $k=1$, we could take $t=0$ if $2\divides (\lambda -d)$ and $t=1$ otherwise, for all $d>3$.  This operator would have the advantage of not needing residue class restrictions built into $\lambda$, that is, as long as $d$ is large enough with respect to $k$ (for instance, $d > H_{all}(k)$ in the notation of \cite{JB}), we could consider all $\lambda$.  The disadvantage, of course, is that the residue class restrictions are built into the asymptotic, and to the best of our knowledge there is no known linear theory for this operator.  This study would be interesting to pursue in the future. 

\begin{proof}
(Of Theorem \ref{WG}). Once the care has been taken regarding the progressions $\Gamma$, the proof follows in a similar fashion as previous examples.  First let $\eta = \lambda - |\bm{p}|^k$ and slice 
\[
A_\lambda (f,g)({\bf x}) := \frac{1}{\lambda^{d/k}}\sum_{|{\bf p}|^k\leq \lambda} (\log{\bf p})f({\bf x-p})\cdot \frac{1}{\lambda^{d/k-1}}\sum_{|{\bf q}|^k=\eta}(\log{\bf q})  g({\bf x-q})
\]
and note that 
\[
\frac{1}{\lambda^{d/k}}\sum_{|{\bf p}|^k\leq \lambda} (\log{\bf p})f({\bf x-p}) \leq \sup_{\lambda \in \Gamma_{d,k}}\frac{1}{\lambda^{d/k}}\sum_{|{\bf p}|^k\leq \lambda} (\log{\bf p})f({\bf x-p}) = M_{HL}^{primes}(f)
\]
so altogether
\[
A_\lambda (f,g)({\bf x}) \leq M_{HL}^{primes}(f)\cdot \sup_{\eta \in \Gamma_{d,k}} \frac{1}{\eta^{d/k-1}}\sum_{|{\bf q}|^k=\eta}(\log{\bf q})  g({\bf x-q}) = M_{HL}^{primes}(f)\cdot A^*(g).
\]
Since $M_{HL}^{primes}$ is bounded on $l^p(\Z^d)$ for all $p> 1$ \cite{Trojan, Wierdl} and $A^*(g)$ is bounded on $l^p(\Z^d)$ for all $p> p_{k,d}$, the result follows in the blinear case.  Induction concludes the $\ell$-linear case.  The necessary conditions yet again follow by testing Dirac delta functions.
\end{proof}

\section{General formulation and further remarks}
In examining these examples, we can see some common elements that are required.  These include: a nice asymptotic for the number of lattice points that scales conveniently, and allows us to ``pull off" the asymptotic for the number of points inside a ball, an additive structure between the $\ell$ variables defining the surface (namely, no mixing of variables via product conditions), positivity in the variables, no ``holes" in the sequence of $\lambda$ defining the surface (or if there are ``holes", that these relate to the holes in the linear versions), and a nice protectivization of the surface in $\Z^{\ell d}$ to the surface in $\Z^{(\ell -1)d}$.  

Each of these elements requires some thought.  For instance, the asymptotic requirement can be seen by the following argument.  Notice that the examples above require an asymptotic for the number of points on our degree $k$ surface in $Z^{\ell d}$ of $\lambda^{\phi(\ell d)}$ and the following relationship:
\[
\lambda^{\phi(\ell d)} = \lambda^{d/k}\cdot \lambda^{\phi((\ell - 1)d)}.
\]
In other words, we need
\[
\frac{\phi(\ell d) - \phi((\ell-1)d)}{d} = \frac{1}{k}
\]
which is a condition on the derivative in the discrete setting (that is, we want this to hold in all large dimensions $d$ and all $\ell \in \N$).  This condition is implied by $\phi'(x) = \frac{1}{k}$, that is $\phi(d) = d/k +C$.  These $\phi$ are precisely the types of functions involved in the asymptotics for lattice points on spheres, balls and related surfaces, This implication clearly indicates the types of surfaces that we can handle by our method based on their asymptotics.  There are also other criteria to consider, as mentioned above.  As we did for the asymptotic, we can phrase these elements more precisely below, allowing us to state a more general theorem about the applicability of the slicing method.  

For convenience, we state everything in terms of the bilinear setup, using the vairables $u$ and $v$.
Consider the following conditions on a bilinear operator
\begin{equation}
\label{general operator}
    T_\lambda(f,g)(\bm{x}) = \left| \frac{1}{N(\lambda)}\sum_{h(\bm{u}, \bm{v},\lambda)}f(\bm{x}-\bm{u})g (\bm{x}-\bm{v}) \right|
\end{equation}
that satisfy
\begin{enumerate}
    \item An asymptotic scaling that allows for a separation into the ball asymptotic and the asymptotic for the underlying surface in $\Z^d$: $N(\lambda) \sim \lambda^{\phi(2d)}$ where $\phi(d) = d/k+C$
    \item A positive, additive structure of $h(\bm{u}, \bm{v}, \lambda) := h(\bm{u}, \bm{v}) \leq \lambda$ (or $h(\bm{u}, \bm{v}, \lambda) := h(\bm{u}, \bm{v}) = \lambda$): that is, $h(\bm{u}, \bm{v}) = h_1(\bm{u})+h_2(\bm{v})$, with $h_1, h_2 \geq 0$
    \item All large $\lambda$ are covered in the asymptotics in both $\Z^{2d}$ and $\Z^d$, or if certain $\lambda$ are excluded, these relate in a one to one fashion to those $\lambda$ excluded in the linear projections (see next item for more details)
    \item There is a one to one correspondence between allowable $\lambda$ and $\eta_1$, where $\eta_1$ is defined via the projection $h_1(\bm{u}) \leq \lambda - h_2(\bm{v}) :=\eta_1$ (or with $``="$ for surfaces $h(\bm{u}, \bm{v}) = \lambda$) onto $\bm{u}$ of $h$ (implicitly assume projection is well-defined in $\Z^d$)
    \item There is a one to one correspondence between allowable $\lambda$ and $\eta_2$, where $\eta_1$ is defined via the projection $h_2(\bm{v}) \leq  \lambda - h_1(\bm{u}) :=\eta_2$ onto $\bm{v}$ of $h$ (the symmetric condition).
\end{enumerate}
Now we can state
\begin{thm}
  Assume conditions 1-5 above.  Also assume that the linear operator
  \[
  T^*(f) := \sup_{\eta_1} \frac{1}{\eta_1^{\phi(d)}}\sum_{h(\bm{u},\eta_1)}\left|f (\bm{x}-\bm{u}) \right|
  \]
 is bounded on $l^p$ for all $p> p_d$.  Then the bilinear operator $T^*(f,g)$ is bounded on $l^{p}(\Z^d)\times l^{q}(\Z^d) \to l^{r}(\Z^d)$, $\frac{1}{p} + \frac{1}{q} \geq \frac{1}{r}$, $r>\frac{p_d}{p_d + 1}$, $p,q> 1$.
\end{thm}
Analogous $\ell$-linear extensions are also possible.  This result underscores that the way that degree $k$ homogeneous (positive) surfaces interact with the additive $\ell$-linear structure is integral to the slicing method. 

\bibliographystyle{amsplain}

\end{document}